\documentclass{amsart}
\usepackage{setspace}
\usepackage{a4}
\usepackage{amsthm}
\usepackage{latexsym}
\usepackage{amsfonts}
\usepackage{graphicx}
\usepackage{textcomp}
\usepackage{cite}
\usepackage{enumerate}
\usepackage{amssymb}
\usepackage{hyperref}
\usepackage{amsmath}
\usepackage{tikz}
\usepackage[mathscr]{euscript}
\usepackage{mathtools}
\newtheorem{theorem}{Theorem}[section]

\title{This is the title}
\usepackage{fancyhdr}

\begin{document}
\hrule\hrule\hrule\hrule\hrule
\vspace{0.3cm}	
\begin{center}
{\bf\large{{Non-Archimedean Brauer Oval (of Cassini) Theorem and Applications}}}\\
\vspace{0.3cm}
\hrule\hrule\hrule\hrule\hrule
\vspace{0.3cm}
\textbf{K. Mahesh Krishna}\\
School of Mathematics and Natural Sciences\\
Chanakya University Global Campus\\
NH-648, Haraluru Village\\
Devanahalli Taluk, 	Bengaluru  North District\\
Karnataka State, 562 110, India\\
Email: kmaheshak@gmail.com\\

Date: \today
\end{center}

\hrule\hrule
\vspace{0.5cm}
\textbf{Abstract}: Nica and Sprague [\textit{Am. Math. Mon., 2023}]
 derived a non-Archimedean version of the Gershgorin disk theorem.  We derive a non-Archimedean version of  the oval (of Cassini)  theorem by  Brauer [\textit{Duke Math. J., 1947}] which generalizes the Nica-Sprague disk theorem. We provide applications for bounding the zeros of polynomials over non-Archimedean fields. We also show that our result is equivalent to the non-Archimedean version of the Ostrowski nonsingularity theorem derived by Li and Li [\textit{J. Comput. Appl. Math., 2025}]. 


\hrule

\hrule
\section{Introduction}
Let $A=[a_{j,k}]_{1\leq j\leq n, 1\leq k \leq n} \in \mathbb{M}_n(\mathbb{C})$. For $1\leq j \leq n$, define 
\begin{align*}
	r_j(A)\coloneqq\sum_{k=1, k \neq j}^{n}|a_{j,k}|.
\end{align*}
For $1\leq k \leq n$, define 
\begin{align*}
	c_k(A)\coloneqq\sum_{j=1, j \neq k}^{n}|a_{j,k}|.
\end{align*}
Let $\sigma(A)$ be the set of all eigenvalues of $A$. In 1931, Gershgorin proved the following breakthrough result known as the Gershgorin circle/disk theorem which uses single row/column  for determining the radius of the disk.
\begin{theorem} \cite{GERSHGORIN, HORNJOHNSON, VARGA, BRUALDIMELLENDORF} \label{GEIT} (\textbf{Gershgorin Eigenvalue Inclusion Theorem} or \textbf{Gershgorin Disk Theorem})
	For every $A=[a_{j,k}]_{1\leq j\leq n, 1\leq k \leq n}\in  \mathbb{M}_n(\mathbb{C})$,
	\begin{align*}
		\sigma(A)\subseteq \bigcup_{j=1}^n\{z \in \mathbb{C}: |z-a_{j, j}|\leq r_j(A)\}
	\end{align*}
and 
\begin{align*}
		\sigma(A)\subseteq \bigcup_{k=1}^n\{z \in \mathbb{C}: |z-a_{k, k}|\leq c_k(A)\}. 
\end{align*}
\end{theorem}
A remarkable application of Theorem \ref{GEIT} is the following result of Frobenius (which advances result of Browne). 
\begin{theorem} \label{FROBENIUST} \cite{BRAUER0, MARCUSMINC, BROWNE}
Let $A=[a_{j,k}]_{1\leq j\leq n, 1\leq k \leq n} \in \mathbb{M}_n(\mathbb{C})$. For every 	$\lambda \in \sigma(A)$, 
\begin{align*}
	|\lambda| \leq \min \left\{\max_{1\leq j \leq n}\sum_{k=1}^{n}|a_{j,k}|, \max_{1\leq k \leq n}\sum_{j=1}^{n}|a_{j,k}|\right\}\leq \frac{1}{2} \left(\max_{1\leq j \leq n}\sum_{k=1}^{n}|a_{j,k}|+ \max_{1\leq k \leq n}\sum_{j=1}^{n}|a_{j,k}|\right).
\end{align*}
In particular, 
\begin{align*}
	|\text{det}(A)| &\leq \min \left\{\left(\max_{1\leq j \leq n}\sum_{k=1}^{n}|a_{j,k}|\right)^n, \left(\max_{1\leq k \leq n}\sum_{j=1}^{n}|a_{j,k}|\right)^n\right\}\\
	&\leq \frac{1}{2} \left(\left(\max_{1\leq j \leq n}\sum_{k=1}^{n}|a_{j,k}|\right)^n+ \left(\max_{1\leq k \leq n}\sum_{j=1}^{n}|a_{j,k}|\right)^n\right).
\end{align*}
\end{theorem}
Another  remarkable application of Theorem \ref{GEIT} is on the bounds for the zeros of polynomials. Let $p(z)\coloneqq c_0+c_1z+\cdots +c_{n-1}z^{n-1}+z^n \in \mathbb{C}[z]$. A direct observation reveals that the zeros of $p$ are the eigenvalues of the Frobenius companion matrix 
\begin{align*}
	C_p\coloneqq 
	\begin{pmatrix}
		0&1&0&\cdots &0 &0&0\\
		0&0&1&\cdots &0 &0&0\\
		0&0&0&\cdots &0 &0&0\\
		\vdots &\vdots &\vdots & & \vdots &\vdots &\vdots\\
		0&0&0&\cdots &0 &1&0\\
		0&0&0&\cdots &0 &0&1\\
	-c_0&-c_1&-c_2&\cdots &-c_{n-3} &-c_{n-2}&-c_{n-1}\\
	\end{pmatrix} \in \mathbb{M}_n(\mathbb{C})
\end{align*}
and the eigenvalues of $C_p$ are the same as the zeros of $p$ \cite{BRAND, BRAND2}. In 1965, Bell derived following bounds for the zeros of $p$ using Theorem \ref{GEIT}.
\begin{theorem} \cite{BELL}
Let $p(z)\coloneqq c_0+c_1z+\cdots +c_{n-1}z^{n-1}+z^n \in \mathbb{C}[z]$.  If $\lambda$ is a  zero of $p$, then 
\begin{align*}
	|\lambda|\leq 1
\end{align*}
or 
\begin{align*}
	|\lambda+c_{n-1}|\leq |c_0|+\cdots +|c_{n-2}|.
\end{align*}
In particular, 
\begin{align}\label{LB}
\text{(Lagrange bound)} \quad \quad \quad 	|\lambda|\leq \max\{1, |c_0|+\cdots+|c_{n-1}|\}.
\end{align}
\end{theorem}
Note that Inequality (\ref{LB}) is a generalization of famous Montel bound \cite{HORNJOHNSON} which says that 
\begin{align*}
	|\lambda|\leq 1+|c_0|+\cdots+|c_{n-1}|.
\end{align*}
It is interesting to note that one can give a proof of Inequality (\ref{LB}) without using companion matrix \cite{HIRSTMACEY}.
\begin{theorem} \cite{BELL}
	Let $p(z)\coloneqq c_0+c_1z+\cdots +c_{n-1}z^{n-1}+z^n \in \mathbb{C}[z]$.  If $\lambda$ is a  zero of $p$, then 
	\begin{align*}
		|\lambda|\leq |c_0|
	\end{align*}
	or 
		\begin{align*}
		|\lambda|\leq 1+|c_{j}|, \quad \text{ for some } 1\leq j \leq n-2
	\end{align*}
	or 
	\begin{align*}
		|\lambda+c_{n-1}|\leq 1.
	\end{align*}
	In particular, 
	\begin{align}\label{BB}
	\text{(Bell bound)} \quad \quad \quad	|\lambda|\leq \max\{|c_0|, 1+|c_1|, \dots, 1+|c_{n-1}|\}.
	\end{align}
\end{theorem}
Note that Inequality (\ref{BB}) is a generalization of famous Cauchy bound \cite{MARDEN} which says that 
\begin{align*}
		|\lambda|\leq 1+\max\{|c_0|, \dots, |c_{n-1}|\}.
\end{align*}
Let $p(z)\coloneqq c_0+c_1z+\cdots +c_{n-1}z^{n-1}+z^n \in \mathbb{C}[z]$ with $c_0\neq 0$. Define 
\begin{align*}
	q(z)\coloneqq \frac{1}{c_0}z^np\left(\frac{1}{z}\right)=\frac{1}{c_0}+\frac{c_{n-1}}{c_0}z+\cdots +\frac{c_{1}}{c_0}z^{n-1}+z^n \in  \mathbb{C}[z].
\end{align*}
We note that $\lambda \in \mathbb{C}\setminus \{0\}$ satisfies $p(\lambda)=0$ if and only if $q(1/\lambda)=0$. Frobenius companion matrix of $q$ is 
\begin{align*}
	C_q\coloneqq 
	\begin{pmatrix}
		0&1&0&\cdots &0 &0&0\\
		0&0&1&\cdots &0 &0&0\\
		0&0&0&\cdots &0 &0&0\\
		\vdots &\vdots &\vdots & & \vdots &\vdots &\vdots\\
		0&0&0&\cdots &0 &1&0\\
		0&0&0&\cdots &0 &0&1\\
		\frac{-1}{c_0}&	\frac{-c_{n-1}}{c_0}&\frac{-c_{n-2}}{c_0}&\cdots &\frac{-c_3}{c_0}& \frac{-c_{2}}{c_0}&\frac{-c_{1}}{c_0}\\
	\end{pmatrix} \in \mathbb{M}_n(\mathbb{C}).
\end{align*}
Now by applying earlier two results to the polynomial $q$ and rearranging, we get following results. 
\begin{theorem} \cite{BELL, HORNJOHNSON}
	Let $p(z)\coloneqq c_0+c_1z+\cdots +c_{n-1}z^{n-1}+z^n \in \mathbb{C}[z]$ with $c_0\neq 0$.  If $\lambda$ is a  zero of $p$, then 
	\begin{align*}
		\frac{1}{|\lambda|}\leq 1
	\end{align*}
	or 
	\begin{align*}
		\left|\frac{1}{\lambda}+\frac{c_1}{c_0}\right|\leq\frac{1}{|c_0|}+\frac{|c_2|}{|c_0|}+\cdots +\frac{|c_{n-1}|}{|c_0|}.
	\end{align*}
	In particular, 
	\begin{align*}
			\frac{1}{|\lambda|} \leq \max\left\{1,  \frac{1}{|c_0|}+\frac{|c_1|}{|c_0|}+\frac{|c_2|}{|c_0|}+\cdots +\frac{|c_{n-1}|}{|c_0|}\right\}\leq 1+\frac{1}{|c_0|}+\frac{|c_1|}{|c_0|}+\cdots +\frac{|c_{n-1}|}{|c_0|}
	\end{align*}
and 
	\begin{align*}
	\text{(Lagrange lower bound)} \quad \quad \quad 	|\lambda|&\geq \frac{|c_0|}{\max\{|c_0|, 1+|c_1|+\cdots+|c_{n-1}|\}}\\
	\text{(Montel lower bound)} \quad \quad \quad	&\geq \frac{|c_0|}{1+|c_0|+|c_1|+\cdots+|c_{n-1}|}.
\end{align*}
\end{theorem}
\begin{theorem} \cite{BELL, HORNJOHNSON}
	Let $p(z)\coloneqq c_0+c_1z+\cdots +c_{n-1}z^{n-1}+z^n \in \mathbb{C}[z]$ with $c_0\neq 0$.  If $\lambda$ is a  zero of $p$, then 	
		\begin{align*}
	\frac{1}{|\lambda|}\leq \frac{1}{|c_0|}
	\end{align*}
	or 
	\begin{align*}
		\frac{1}{|\lambda|}\leq 1+\frac{|c_j|}{|c_0|}, \quad \text{ for some } 2\leq j \leq n-1
	\end{align*}
	or 
	\begin{align*}
		\left|\frac{1}{\lambda}+\frac{c_1}{c_0}\right|\leq 1.
	\end{align*}
	In particular, 
	\begin{align*}
		\frac{1}{|\lambda|}\leq \max\left\{\frac{1}{|c_0|}, 1+\frac{|c_1|}{|c_0|}, \dots, 1+\frac{|c_{n-1}|}{|c_0|}\right\}\leq 1+\max\left\{\frac{1}{|c_0|}, \frac{|c_1|}{|c_0|}, \dots, \frac{|c_{n-1}|}{|c_0|}\right\}
	\end{align*}
	and 
	\begin{align*}
		\text{(Bell lower bound)} \quad \quad \quad	|\lambda|&\geq \frac{|c_0|}{\max\{1, |c_0|+|c_1|, \dots, |c_0|+|c_{n-1}|\}}\\
		\text{(Cauchy lower bound)} \quad \quad \quad 	&\geq \frac{|c_0|}{|c_0|+\max\{1, |a_1|, \dots, |a_{n-1}|\}}.
	\end{align*}
\end{theorem}
A result much older than Gershgorin is the following. 
\begin{theorem} \cite{VARGA, TAUSSKY, HORNJOHNSON, SCHNEIDER} \label{SDDT} (\textbf{Strict Diagonal Dominance Theorem} or \textbf{Levy-Desplanques Theorem})
	If  $A=[a_{j,k}]_{1\leq j\leq n, 1\leq k \leq n}\in  \mathbb{M}_n(\mathbb{C})$ satisfies
	\begin{align*}
		|a_{j,j}|>r_j(A), \quad \forall 1\leq j \leq n
	\end{align*}
	or 
	\begin{align*}
		|a_{k,k}|>c_k(A), \quad \forall 1\leq k \leq n,
	\end{align*}
	then $A$ is invertible. 
\end{theorem}
Following results say that Theorems \ref{GEIT} and \ref{SDDT} are equivalent.
\begin{theorem} \cite{VARGA}
	Let $n \in \mathbb{N}$.  Following two statements are equivalent. 
	\begin{enumerate}[\upshape(i)]
		\item Let $A=[a_{j,k}]_{1\leq j\leq n, 1\leq k \leq n}\in  \mathbb{M}_n(\mathbb{C})$. Then 
		\begin{align*}
			\sigma(A)\subseteq \bigcup_{j=1}^n\{z \in \mathbb{C}: |z-a_{j, j}|\leq r_j(A)\}.
		\end{align*}
		\item 	If  $B=[b_{j,k}]_{1\leq j\leq n, 1\leq k \leq n}\in  \mathbb{M}_n(\mathbb{C})$ satisfies
		\begin{align*}
			|b_{j,j}|>r_j(B), \quad \forall 1\leq j \leq n,
		\end{align*}
		then $B$ is invertible. 
	\end{enumerate}	
\end{theorem}
\begin{theorem} \cite{VARGA}
	Let $n \in \mathbb{N}$.  Following two statements are equivalent. 
	\begin{enumerate}[\upshape(i)]
		\item Let $A=[a_{j,k}]_{1\leq j\leq n, 1\leq k \leq n}\in  \mathbb{M}_n(\mathbb{C})$. Then 
		\begin{align*}
			\sigma(A)\subseteq \bigcup_{k=1}^n\{z \in \mathbb{C}: |z-a_{k,k}|\leq c_k(A)\}.
		\end{align*}
		\item 	If  $B=[b_{j,k}]_{1\leq j\leq n, 1\leq k \leq n}\in  \mathbb{M}_n(\mathbb{C})$ satisfies
		\begin{align*}
			|b_{k,k}|>c_k(B), \quad \forall 1\leq k \leq n,
		\end{align*}
		then $B$ is invertible. 
	\end{enumerate}	
\end{theorem}
In 1947, Brauer derived following generalization of Theorem \ref{GEIT}, known as the Brauer oval (of Cassini) theorem which uses two rows/columns  for determining the radius of the oval.
\begin{theorem} \cite{BRAUER, VARGA, MARCUSMINC} \label{BET} (\textbf{Brauer Eigenvalue Inclusion Theorem} or \textbf{Brauer Oval (of Cassini) Theorem})
	For every $A=[a_{j,k}]_{1\leq j\leq n, 1\leq k \leq n}\in  \mathbb{M}_n(\mathbb{C})$,
\begin{align*}
	\sigma(A)\subseteq \bigcup_{j,k=1, j \neq k}^n\{z \in \mathbb{C}: |z-a_{j, j}||z-a_{k,k}|\leq r_j(A)r_k(A)\}
\end{align*}
and 
\begin{align*}
	\sigma(A)\subseteq \bigcup_{j,k=1, j \neq k}^n\{z \in \mathbb{C}: |z-a_{j, j}||z-a_{k,k}|\leq c_j(A)c_k(A)\}.
\end{align*}
\end{theorem}
By applying Brauer's theorem, we get following results.
\begin{theorem} 
	Let $p(z)\coloneqq c_0+c_1z+\cdots +c_{n-1}z^{n-1}+z^n \in \mathbb{C}[z]$.  If $\lambda$ is a  zero of $p$, then 
	\begin{align*}
		|\lambda|\leq 1
	\end{align*}
	or 
	\begin{align*}
		|\lambda||\lambda+c_{n-1}|\leq |c_0|+\cdots +|c_{n-2}|.
	\end{align*}
\end{theorem}
\begin{theorem} 
	Let $p(z)\coloneqq c_0+c_1z+\cdots +c_{n-1}z^{n-1}+z^n \in \mathbb{C}[z]$.  If $\lambda$ is a  zero of $p$, then 
\begin{align*}
	|\lambda|^2\leq |c_0|(1+|c_j|),  \quad \text{ for some } 1\leq j \leq n-2
\end{align*}
or 
\begin{align*}
	|\lambda||\lambda +c_{n-1}|\leq |c_0|
\end{align*}
or 
\begin{align*}
	|\lambda|^2\leq (1+|c_{j}|)(1+|c_{k}|), \quad \text{ for some } 1\leq j, k \leq n-2, j \neq k
\end{align*}
or 
\begin{align*}
	|\lambda||\lambda +c_{n-1}|\leq 1+|c_{j}|, \quad \text{ for some } 1\leq j \leq n-2.
\end{align*}
\end{theorem}
It is known that Brauer theorem cannot be extended by considering three rows/columns \cite{VARGA}.  
In 1937, Ostrowski derived following generalization of Theorem \ref{SDDT}.
\begin{theorem} \cite{VARGA, OSTROWSKI} \label{ONT} (\textbf{Ostrowski Nonsingularity Theorem})
	If  $A=[a_{j,k}]_{1\leq j\leq n, 1\leq k \leq n}\in  \mathbb{M}_n(\mathbb{C})$ satisfies
	\begin{align*}
		|a_{j,j}||a_{k,k}|>r_j(A)r_k(A), \quad \forall 1\leq j, k \leq n, j \neq k
	\end{align*}
	or 
	\begin{align*}
		|a_{j,j}||a_{k,k}|>c_j(A)c_k(A), \quad \forall 1\leq j, k \leq n, j \neq k,
	\end{align*}
	then $A$ is invertible. 	
\end{theorem}
It is known that Theorems \ref{BET} and \ref{ONT} are again equivalent. 
\begin{theorem} \cite{VARGA}
	Let $n \in \mathbb{N}$.  Following two statements are equivalent. 
	\begin{enumerate}[\upshape(i)]
		\item Let $A=[a_{j,k}]_{1\leq j\leq n, 1\leq k \leq n}\in  \mathbb{M}_n(\mathbb{C})$. Then 
		\begin{align*}
			\sigma(A)\subseteq \bigcup_{j,k=1, j \neq k}^n\{z \in \mathbb{C}: |z-a_{j, j}||z-a_{k,k}|\leq r_j(A)r_k(A)\}.
		\end{align*}
		\item 	If  $B=[b_{j,k}]_{1\leq j\leq n, 1\leq k \leq n}\in  \mathbb{M}_n(\mathbb{C})$ satisfies
		\begin{align*}
			|b_{j,j}||b_{k,k}|>	r_j(B)r_k(B), \quad \forall 1\leq j, k \leq n, j \neq k,
		\end{align*}
		then $B$ is invertible. 
	\end{enumerate}	
\end{theorem}
\begin{theorem} \cite{VARGA}
	Let $n \in \mathbb{N}$.  Following two statements are equivalent. 
	\begin{enumerate}[\upshape(i)]
		\item Let $A=[a_{j,k}]_{1\leq j\leq n, 1\leq k \leq n}\in  \mathbb{M}_n(\mathbb{C})$. Then 
		\begin{align*}
			\sigma(A)\subseteq \bigcup_{j,k=1, j \neq k}^n\{z \in \mathbb{C}: |z-a_{j, j}||z-a_{k,k}|\leq c_j(A)c_k(A)\}.
		\end{align*}
		\item 	If  $B=[b_{j,k}]_{1\leq j\leq n, 1\leq k \leq n}\in  \mathbb{M}_n(\mathbb{C})$ satisfies
		\begin{align*}
			|b_{j,j}||b_{k,k}|>	c_j(B)c_k(B), \quad \forall 1\leq j, k \leq n, j \neq k,
		\end{align*}
		then $B$ is invertible. 
	\end{enumerate}	
\end{theorem}
It is natural to ask what are versions of Theorems \ref{GEIT}, \ref{FROBENIUST}, \ref{SDDT}, \ref{BET} and \ref{ONT}  for matrices over non-Archimedean fields? 
In the paper,  $\mathbb{K}$ denotes a non-Archimedean valued field with valuation $|\cdot|$. Let $A=[a_{j,k}]_{1\leq j\leq n, 1\leq k \leq n} \in \mathbb{M}_n(\mathbb{K})$. For $1\leq j \leq n$, define 
\begin{align*}
	h_j(A)\coloneqq\max_{1\leq k \leq n, k \neq j}|a_{j,k}|.
\end{align*}
For $1\leq k \leq n$, define 
\begin{align*}
	v_k(A)\coloneqq\max_{1\leq j \leq n, j \neq k}|a_{j,k}|.
\end{align*}

In 2023, Nica and Sprague derived the  following non-Archimedean analogue of Theorems \ref{GEIT} and \ref{SDDT}. 
\begin{theorem} \label{NS}  \cite{NICASPRAGUE} (\textbf{Non-Archimedean Gershgorin Eigenvalue Inclusion Theorem} or \textbf{Nica-Sprague Disk Theorem}) 	For every $A=[a_{j,k}]_{1\leq j\leq n, 1\leq k \leq n}\in  \mathbb{M}_n(\mathbb{K})$,
	\begin{align*}
		\sigma(A)\subseteq \bigcup_{j=1}^n\{z \in \mathbb{K}: |z-a_{j, j}|\leq h_j(A)\}
	\end{align*}
	and 
	\begin{align*}
		\sigma(A)\subseteq \bigcup_{k=1}^n\{z \in \mathbb{K}: |z-a_{k, k}|\leq v_k(A)\}. 
	\end{align*}
\end{theorem}
\begin{theorem} \label{NS2} \cite{NICASPRAGUE} (\textbf{Non-Archimedean Strict Diagonal Dominance Theorem} or \textbf{Nica-Sprague Nonsingularity Theorem}) 
	If  $A=[a_{j,k}]_{1\leq j\leq n, 1\leq k \leq n}\in  \mathbb{M}_n(\mathbb{K})$ satisfies
	\begin{align*}
		|a_{j,j}|>h_j(A), \quad \forall 1\leq j \leq n
	\end{align*}
	or 
	\begin{align*}
		|a_{k,k}|>v_k(A), \quad \forall 1\leq k \leq n,
	\end{align*}
	then $A$ is invertible. 	
\end{theorem}
Nica and Sprague showed that  Theorems \ref{NS} and \ref{NS2} are equivalent. 
\begin{theorem} \cite{NICASPRAGUE}
	Let $n \in \mathbb{N}$.  Following two statements are equivalent. 
	\begin{enumerate}[\upshape(i)]
		\item Let $A=[a_{j,k}]_{1\leq j\leq n, 1\leq k \leq n}\in  \mathbb{M}_n(\mathbb{K})$. Then 
		\begin{align*}
			\sigma(A)\subseteq \bigcup_{j=1}^n\{z \in \mathbb{K}: |z-a_{j, j}|\leq h_j(A)\}.
		\end{align*}
		\item 	If  $B=[b_{j,k}]_{1\leq j\leq n, 1\leq k \leq n}\in  \mathbb{M}_n(\mathbb{K})$ satisfies
		\begin{align*}
			|b_{j,j}|>h_j(B), \quad \forall 1\leq j \leq n,
		\end{align*}
		then $B$ is invertible. 
	\end{enumerate}	
\end{theorem}
\begin{theorem} \cite{NICASPRAGUE}
	Let $n \in \mathbb{N}$.  Following two statements are equivalent. 
	\begin{enumerate}[\upshape(i)]
		\item Let $A=[a_{j,k}]_{1\leq j\leq n, 1\leq k \leq n}\in  \mathbb{M}_n(\mathbb{K})$. Then 
		\begin{align*}
			\sigma(A)\subseteq \bigcup_{k=1}^n\{z \in \mathbb{K}: |z-a_{k,k}|\leq v_k(A)\}.
		\end{align*}
		\item 	If  $B=[b_{j,k}]_{1\leq j\leq n, 1\leq k \leq n}\in  \mathbb{M}_n(\mathbb{K})$ satisfies
		\begin{align*}
			|b_{k,k}|>v_k(B), \quad \forall 1\leq k \leq n,
		\end{align*}
		then $B$ is invertible. 
	\end{enumerate}	
\end{theorem}
Non-Archimedean version of Theorem \ref{FROBENIUST} reads as follows. 
\begin{theorem}
Let $A=[a_{j,k}]_{1\leq j\leq n, 1\leq k \leq n} \in \mathbb{M}_n(\mathbb{K})$. For every 	$\lambda \in \sigma(A)$, 
	\begin{align*}
		|\lambda| \leq \min \left\{\max_{1\leq j \leq n}\max_{1\leq k \leq n}|a_{j,k}|, \max_{1\leq k \leq n}\max_{1\leq j \leq n}|a_{j,k}|\right\}\leq \frac{1}{2}\left(\max_{1\leq j \leq n}\max_{1\leq k \leq n}|a_{j,k}|+ \max_{1\leq k \leq n}\max_{1\leq j \leq n}|a_{j,k}|\right).
	\end{align*}
	In particular, 
	\begin{align*}
		|\text{det}(A)| &\leq  \min \left\{\left(\max_{1\leq j \leq n}\max_{1\leq k \leq n}|a_{j,k}|\right)^n, \left(\max_{1\leq k \leq n}\max_{1\leq j \leq n}|a_{j,k}|\right)^n\right\}\\
		&\leq \frac{1}{2}\left(\left(\max_{1\leq j \leq n}\max_{1\leq k \leq n}|a_{j,k}|\right)^n+ \left(\max_{1\leq k \leq n}\max_{1\leq j \leq n}|a_{j,k}|\right)^n\right).
	\end{align*}
\end{theorem}	
Let $p(z)\coloneqq c_0+c_1z+\cdots +c_{n-1}z^{n-1}+z^n \in \mathbb{K}[z]$. Let 
\begin{align*}
	C_p\coloneqq 
	\begin{pmatrix}
		0&1&0&\cdots &0 &0&0\\
		0&0&1&\cdots &0 &0&0\\
		0&0&0&\cdots &0 &0&0\\
		\vdots &\vdots &\vdots & & \vdots &\vdots &\vdots\\
		0&0&0&\cdots &0 &1&0\\
		0&0&0&\cdots &0 &0&1\\
		-c_0&-c_1&-c_2&\cdots &-c_{n-3} &-c_{n-2}&-c_{n-1}\\
	\end{pmatrix} \in \mathbb{M}_n(\mathbb{K})
\end{align*}
be the companion matrix of $p$. By applying Theorem \ref{NS} Nica and Sprague obtained following results.
\begin{theorem} \cite{NICASPRAGUE}
	Let $p(z)\coloneqq c_0+c_1z+\cdots +c_{n-1}z^{n-1}+z^n \in \mathbb{K}[z]$.  If $\lambda$ is a  zero of $p$, then 
	\begin{align*}
		|\lambda|\leq 1
	\end{align*}
	or 
	\begin{align*}
		|\lambda+c_{n-1}|\leq \max\{|c_0|, |c_1|, \dots,  |c_{n-2}|\}.
	\end{align*}
	In particular, 
	\begin{align*}
	\text{(Nica-Sprague bound)} \quad \quad \quad	|\lambda|\leq \max\{1, |c_0|, |c_1|, \dots, |c_{n-1}|\}.
	\end{align*}
\end{theorem}
\begin{theorem} \cite{NICASPRAGUE}
	Let $p(z)\coloneqq c_0+c_1z+\cdots +c_{n-1}z^{n-1}+z^n \in \mathbb{K}[z]$.  If $\lambda$ is a  zero of $p$, then 
	\begin{align*}
		|\lambda|\leq |c_o|
	\end{align*}
	or 
	\begin{align*}
		|\lambda|\leq \max\{1, |c_j|\}, \quad \text{ for some } 1\leq j \leq n-2
	\end{align*}
or
	\begin{align*}
		|\lambda +c_{n-1}|\leq 1.
	\end{align*}
In particular, 
\begin{align}\label{NSB}
	\text{(Nica-Sprague bound)} \quad \quad \quad	|\lambda|\leq \max\{1, |c_0|, |c_1|, \dots, |c_{n-1}|\}.
\end{align}
\end{theorem}
\begin{theorem} \cite{NICASPRAGUE}
	Let $p(z)\coloneqq c_0+c_1z+\cdots +c_{n-1}z^{n-1}+z^n \in \mathbb{K}[z]$ with $c_0\neq 0$.  If $\lambda$ is a  zero of $p$, then 
	\begin{align*}
	\frac{1}{|\lambda|}\leq  1
	\end{align*}
	or 
	\begin{align*}
	\left|\frac{1}{\lambda}+\frac{c_1}{c_0}\right|	\leq \max\left\{\frac{1}{|c_0|}, \frac{|c_2|}{|c_0|}, \dots,  \frac{|c_{n-1}|}{|c_0|}\right\}.
	\end{align*}
	In particular, 
	\begin{align*}
		\text{(Nica-Sprague lower bound)} \quad \quad \quad	|\lambda|\geq  \frac{|c_0|}{\max\{1, |c_0|, |c_1|, \dots, |c_{n-1}|\}}.
	\end{align*}
\end{theorem}
\begin{theorem} \cite{NICASPRAGUE}
	Let $p(z)\coloneqq c_0+c_1z+\cdots +c_{n-1}z^{n-1}+z^n \in \mathbb{K}[z]$ with $c_0\neq 0$.  If $\lambda$ is a  zero of $p$, then 
	\begin{align*}
		\frac{1}{|\lambda|}\leq \frac{1}{|c_0|}
	\end{align*}
	or 
	\begin{align*}
			\frac{1}{|\lambda|}\leq \max\left\{1, \frac{|c_j|}{|c_0|}\right\}, \quad \text{ for some } 2\leq j \leq n-1
	\end{align*}
	or
	\begin{align*}
		\left|\frac{1}{\lambda}+\frac{c_1}{c_0}\right|	\leq 1.
	\end{align*}
	In particular, 
	\begin{align*}
		\text{(Nica-Sprague lower bound)} \quad \quad \quad	|\lambda|\geq \frac{|c_0|}{\max\{1, |c_0|, |c_1|, \dots, |c_{n-1}|\}}.
	\end{align*}
\end{theorem}
Before passing, we give a different and direct proof of Inequality (\ref{NSB}). 	Let $p(z)\coloneqq c_0+c_1z+\cdots +c_{n-1}z^{n-1}+z^n \in \mathbb{K}[z]$ and  $\lambda$ be a  zero of $p$. If $|\lambda|\leq 1$, then clearly we have Inequality (\ref{NSB}). So we assume that $|\lambda|> 1$. Since $p(\lambda)=0$, we have 
\begin{align*}
	\left(\frac{c_0}{\lambda^n}+\frac{c_1}{\lambda^{n-1}}+\cdots +\frac{c_{n-1}}{\lambda}+1\right)\lambda^n=0.
\end{align*}
Since $\lambda \neq 0$, 
\begin{align*}
	\frac{c_0}{\lambda^n}+\frac{c_1}{\lambda^{n-1}}+\cdots +\frac{c_{n-1}}{\lambda}+1=0.
\end{align*}
Rearranging, 
\begin{align*}
	-1=\frac{c_0}{\lambda^n}+\frac{c_1}{\lambda^{n-1}}+\cdots +\frac{c_{n-1}}{\lambda}.
\end{align*}
By taking absolute value and noticing $|\lambda|>1$, we get 
\begin{align*}
	1&\leq \left|\frac{c_0}{\lambda^n}+\frac{c_1}{\lambda^{n-1}}+\cdots +\frac{c_{n-1}}{\lambda}\right|\leq \max\left\{\frac{|c_0|}{|\lambda|^n}, \frac{|c_1|}{|\lambda|^{n-1}}, \dots, \frac{|c_{n-1}|}{|\lambda|}\right\}\\
	&\leq \max\left\{\frac{|c_0|}{|\lambda|}, \frac{|c_1|}{|\lambda|}, \dots, \frac{|c_{n-1}|}{|\lambda|}\right\}=\frac{1}{|\lambda|}\max\{|c_0|, |c_1|, \dots, |c_{n-1}|\}.
\end{align*}
Rearranging above inequality completes the argument. \\
In 2025, Li and Li derived the following non-Archimedean analogue of Theorem  \ref{ONT}.
\begin{theorem} \cite{LILI} \label{LLT}
	(\textbf{Non-Archimedean Ostrowski Nonsingularity Theorem} or \textbf{Li-Li Nonsingularity Theorem})
	If  $A=[a_{j,k}]_{1\leq j\leq n, 1\leq k \leq n}\in  \mathbb{M}_n(\mathbb{K})$ satisfies
	\begin{align*}
		|a_{j,j}||a_{k,k}|>h_j(A)h_k(A), \quad \forall 1\leq j, k \leq n, j \neq k
	\end{align*}
	or 
	\begin{align*}
		|a_{j,j}||a_{k,k}|>v_j(A)v_k(A), \quad \forall 1\leq j, k \leq n, j \neq k,
	\end{align*}
	then $A$ is invertible. 	
\end{theorem}
In this article, we derive non-Archimedean version of Theorem \ref{BET}. We also show that our result is equivalent to Theorem \ref{LLT}. We give applications for bounding the zeros of polynomials over non-Archimedean fields.

\section{Non-Archimedean Brauer Oval (of Cassini)  Theorem}

We start with non-Archimedean Brauer eigenvalue inclusion theorem. Our proof is motivated from the proof of Brauer \cite{BRAUER}.
\begin{theorem}  \label{NAB} (\textbf{Non-Archimedean Brauer Oval (of Cassini) Theorem})

For every $A=[a_{j,k}]_{1\leq j\leq n, 1\leq k \leq n}\in  \mathbb{M}_n(\mathbb{K})$,
\begin{align*}
	\sigma(A)\subseteq \bigcup_{j,k=1, j \neq k}^n\{z \in \mathbb{K}: |z-a_{j, j}||z-a_{k,k}|\leq h_j(A)h_k(A)\}
\end{align*}
and 
\begin{align*}
	\sigma(A)\subseteq \bigcup_{j,k=1, j \neq k}^n\{z \in \mathbb{K}: |z-a_{j, j}||z-a_{k,k}|\leq v_j(A)v_k(A)\}.
\end{align*}	
\end{theorem}
\begin{proof}
Let $\lambda \in \sigma(A)$. Then there exists a $0\neq \textbf{x} =(x_j)_{j=1}^n \in \mathbb{K}^n$ such that 
\begin{align}\label{PE1}
	\lambda \textbf{x} =A\textbf{x}.
\end{align}
Choose $1\leq j \leq n$ such that 
\begin{align*}
	 |x_j|=\max_{1\leq l \leq n}|x_l|.
\end{align*}
Now choose $1\leq k \leq n$ with $k \neq j$ such that 
\begin{align*}
	|x_k|=\max_{1\leq l \leq n, l\neq j}|x_l|.
\end{align*}
Then we have $1\leq j, k \leq n$ with $j \neq k$ and 
\begin{align}\label{PE2}
	|x_j|\geq |x_k|\geq \max_{1\leq l \leq n, l\neq j, l \neq k}|x_l|.
\end{align}
We have two cases. Case (i): $ |x_k|=0$. Considering the $j$-th coordinate in Equation (\ref{PE1}) gives 
\begin{align*}
	\lambda x_j=\sum_{p=1}^na_{j,p}x_p.
\end{align*}
Rewriting previous equation gives 
\begin{align*}
	(\lambda -a_{j,j})x_j=\sum_{p=1, p\neq j}^na_{j,p}x_p.
\end{align*}
Therefore using (\ref{PE2}) we get
\begin{align*}
	|(\lambda -a_{j,j})x_j|&=\left|\sum_{p=1, p\neq j}^na_{j,p}x_p\right|\leq \max_{1\leq p \leq n, p \neq j}|a_{j,p}x_p|\\
	&\leq \left(\max_{1\leq p \leq n, p \neq j}|a_{j,p}|\right)\left(\max_{1\leq p \leq n, p \neq j}|x_p|\right)=\left(\max_{1\leq p \leq n, p \neq j}|a_{j,p}|\right)|x_k|\\
	&=0.
\end{align*}
Since $	|x_j|\geq |x_k|$ for all $1\leq k \leq n$,  $ |x_k|=0$ and $\textbf{x}\neq 0$, we must have $|x_j|\neq 0$. Previous inequality then gives $	|\lambda -a_{j,j}|=0$. So 
\begin{align*}
\lambda =a_{j,j} \in \bigcup_{p, q=1, p \neq q}^n\{z \in \mathbb{K}: |z-a_{p,p}||z-a_{q,q}|\leq h_p(A)h_q(A)\}.
\end{align*}
Case (ii): $ |x_k|>0$. Considering $j$-th and $k$-th coordinates in Equation (\ref{PE1}) give
\begin{align}\label{PE3}
	(\lambda -a_{j,j})x_j=\sum_{p=1, p\neq j}^na_{j,p}x_p
\end{align}
and 
\begin{align}\label{PE4}
	(\lambda -a_{k,k})x_k=\sum_{q=1, q\neq k}^na_{k,q}x_q.
\end{align}
Multiplying Equations (\ref{PE3}) and (\ref{PE4}) and taking non-Archimedean valuation gives
\begin{align*}
|(\lambda -a_{j,j})x_j	(\lambda -a_{k,k})x_k|&=\left|\sum_{p=1, p\neq j}^na_{j,p}x_p\right|\left|\sum_{q=1, q\neq k}^na_{k,q}x_q\right|\\
&\leq \left(\max_{1\leq p \leq n, p \neq j}|a_{j,p}x_p|\right)\left(\max_{1\leq q \leq n, q \neq k}|a_{k,p}x_q|\right)\\
&\leq \left(\max_{1\leq p \leq n, p \neq j}|a_{j,p}|\right)\left(\max_{1\leq p \leq n, p \neq j}|x_p|\right)\left(\max_{1\leq q \leq n, q \neq k}|a_{k,p}|\right)\left(\max_{1\leq q \leq n, q \neq k}|x_q|\right)\\
&=\left(\max_{1\leq p \leq n, p \neq j}|a_{j,p}|\right)|x_k|\left(\max_{1\leq q \leq n, q \neq k}|a_{k,p}|\right)|x_j|\\
&=h_j(A)|x_k|h_k(A)|x_j|.
\end{align*}
Therefore 
\begin{align*}
|(\lambda -a_{j,j})(\lambda -a_{k,k})||x_jx_k|\leq h_j(A)h_k(A)|x_j||x_k|.
\end{align*}
Since $|x_jx_k|\neq0$, we have 
\begin{align*}
|(\lambda -a_{j,j})(\lambda -a_{k,k})|\leq h_j(A)h_k(A).	
\end{align*}
Previous inequality says that 
\begin{align*}
	\lambda  \in \bigcup_{p, q=1, p \neq q}^n\{z \in \mathbb{K}: |z-a_{p,p}||z-a_{q,q}|\leq h_p(A)h_q(A)\}.
\end{align*}
Second inclusion in the statement follows by considering the transpose of $A$ and noting that the spectrum of a matrix and its transpose are equal.	
\end{proof}
By applying Theorem \ref{NAB}, we get following results.
\begin{theorem} 
	Let $p(z)\coloneqq c_0+c_1z+\cdots +c_{n-1}z^{n-1}+z^n \in \mathbb{K}[z]$.  If $\lambda$ is a  zero of $p$, then 
	\begin{align*}
		|\lambda|\leq 1
	\end{align*}
	or 
	\begin{align*}
		|\lambda||\lambda+c_{n-1}|\leq \max\{|c_0|, \cdots ,|c_{n-2}|\}.
	\end{align*}
\end{theorem}
\begin{theorem} 
	Let $p(z)\coloneqq c_0+c_1z+\cdots +c_{n-1}z^{n-1}+z^n \in \mathbb{K}[z]$.  If $\lambda$ is a  zero of $p$, then 
	\begin{align*}
		|\lambda|^2\leq |c_0|\max\{1, |c_j|\},  \quad \text{ for some } 1\leq j \leq n-2
	\end{align*}
	or 
	\begin{align*}
		|\lambda||\lambda +c_{n-1}|\leq |c_0|
	\end{align*}
	or 
	\begin{align*}
		|\lambda|^2\leq \left(\max\{1, |c_j|\}\right)\left(\max\{1, |c_k|\}\right), \quad \text{ for some } 1\leq j, k \leq n-2, j \neq k
	\end{align*}
	or 
	\begin{align*}
		|\lambda||\lambda +c_{n-1}|\leq \max\{1, |c_j|\},  \quad \text{ for some } 1\leq j \leq n-2.
	\end{align*}
\end{theorem}
\begin{theorem} 
	Let $p(z)\coloneqq c_0+c_1z+\cdots +c_{n-1}z^{n-1}+z^n \in \mathbb{K}[z]$ with $c_0\neq 0$.  If $\lambda$ is a  zero of $p$, then 
	\begin{align*}
	\frac{1}{|\lambda|}\leq 1
	\end{align*}
	or 
	\begin{align*}
		\frac{1}{|\lambda|}\left|\frac{1}{\lambda}+\frac{c_1}{c_0}\right|	\leq \max\left\{\frac{1}{|c_0|}, 	\frac{|c_2|}{|c_0|}, \cdots,	\frac{|c_{n-1}|}{|c_0|}\right\}.
	\end{align*}
\end{theorem}
\begin{theorem} 
	Let $p(z)\coloneqq c_0+c_1z+\cdots +c_{n-1}z^{n-1}+z^n \in \mathbb{K}[z]$ with $c_0\neq 0$.  If $\lambda$ is a  zero of $p$, then 
	\begin{align*}
		\frac{1}{|\lambda|^2}\leq 	\frac{1}{|c_0|}\max\left\{1, 	\frac{|c_j|}{|c_0|}\right\},  \quad \text{ for some } 2\leq j \leq n-1
	\end{align*}
	or 
	\begin{align*}
		\frac{1}{|\lambda|}	\left|\frac{1}{\lambda}+\frac{c_1}{c_0}\right|\leq \frac{1}{|c_0|}
	\end{align*}
	or 
	\begin{align*}
		\frac{1}{|\lambda|^2}\leq \left(\max\left\{1, \frac{|c_j|}{|c_0|}\right\}\right)\left(\max\left\{1, \frac{|c_k|}{|c_0|}\right\}\right), \quad \text{ for some } 2\leq j, k \leq n-1, j \neq k
	\end{align*}
	or 
	\begin{align*}
		\frac{1}{|\lambda|}	\left|\frac{1}{\lambda}+\frac{c_1}{c_0}\right|\leq \max\left\{1, \frac{|c_j|}{|c_0|}\right\},  \quad \text{ for some } 2\leq j \leq n-1.
	\end{align*}
\end{theorem}

Like the complex case, Theorem \ref{NAB} cannot be extended by considering three rows/columns. An example given for the scalar case in \cite{VARGA} (also see \cite{NEWMANTHOMPSON, BRUALDI}) extends to non-Archimedean case.  Consider the  following matrix over any non-Archimedean field $\mathbb{K}$. 
\begin{align*}
A\coloneqq 
\begin{pmatrix}
1&1&0&0\\
1&1&0&0\\
0&0&1&0\\
0&0&0&1\\
\end{pmatrix}.
\end{align*}
Then $\sigma(A)=\{0, 1, 1, 2\}$ and $h_1(A)=1$, $h_2(A)=1$, $h_3(A)=0$, $h_4(A)=0$. Hence 
\begin{align*}
	\sigma (A) \nsubseteq \bigcup_{j,k,l=1, j \neq k, j \neq l, k \neq l}^4\{z \in \mathbb{K}: |z-a_{j, j}||z-a_{k,k}||z-a_{l,l}|\leq h_j(A)h_k(A)h_l(A)\}=\{0\}.
\end{align*}
Next we show that Theorem \ref{NAB} improves Theorem \ref{NS}. Our proof is motivated from the complex case, given in \cite{VARGA}.
\begin{theorem}
	For every $A=[a_{j,k}]_{1\leq j\leq n, 1\leq k \leq n}\in  \mathbb{M}_n(\mathbb{K})$, 
	\begin{align*}
	\bigcup_{j,k=1, j \neq k}^n\{z \in \mathbb{K}: |z-a_{j, j}||z-a_{k,k}|\leq h_j(A)h_k(A)\}\subseteq \bigcup_{j=1}^n\{z \in \mathbb{K}: |z-a_{j, j}|\leq h_j(A)\}
	\end{align*}
and 
	\begin{align*}
	\bigcup_{j,k=1, j \neq k}^n\{z \in \mathbb{K}: |z-a_{j, j}||z-a_{k,k}|\leq v_j(A)v_k(A)\}\subseteq \bigcup_{j=1}^n\{z \in \mathbb{K}: |z-a_{j, j}|\leq v_j(A)\}.
\end{align*}
\end{theorem}
\begin{proof}
We prove the first inclusion, proof of second inclusion is similar. Set 
\begin{align*}
Z\coloneqq \bigcup_{j,k=1, j \neq k}^n\{z \in \mathbb{K}: |z-a_{j, j}||z-a_{k,k}|\leq h_j(A)h_k(A)\}
\end{align*}	
and let $z \in Z$. Then there exist $1\leq j, k \leq n$ with $j \neq k$ such that 
\begin{align*}
	|z-a_{j, j}||z-a_{k,k}|\leq h_j(A)h_k(A).
\end{align*}
We have to consider two cases. Case (i): $h_j(A)h_k(A)=0$. Then $z=a_{j, j}$ or $z=a_{k, k}$. Now clearly we have 
\begin{align*}
	z \in \{x \in \mathbb{K}: |x-a_{j,j}|\leq h_j(A)\} \cup \{y \in \mathbb{K}: |y-a_{k, k}|\leq h_k(A)\} \subseteq \bigcup_{l=1}^n\{z \in \mathbb{K}: |z-a_{l, l}|\leq v_l(A)\}.
\end{align*}
Case(ii): $h_j(A)h_k(A)>0$. Then 
\begin{align*}
	|z-a_{j, j}|\leq h_j(A)	
\end{align*}
or 
\begin{align*}
	|z-a_{k,k}|\leq h_k(A).
\end{align*}
Now clearly we have 
\begin{align*}
	z \in \{x \in \mathbb{K}: |x-a_{j,j}|\leq h_j(A)\} \cup \{y \in \mathbb{K}: |y-a_{k, k}|\leq h_k(A)\} \subseteq \bigcup_{l=1}^n\{z \in \mathbb{K}: |z-a_{l, l}|\leq v_l(A)\}.
\end{align*}
\end{proof}
Now we show that Theorems \ref{LLT} and \ref{NAB} are equivalent.
\begin{theorem}\label{E1}
	Let $n \in \mathbb{N}$.  Following two statements are equivalent. 
	\begin{enumerate}[\upshape(i)]
		\item Let $A=[a_{j,k}]_{1\leq j\leq n, 1\leq k \leq n}\in  \mathbb{M}_n(\mathbb{K})$. Then 
		\begin{align*}
			\sigma(A)\subseteq \bigcup_{j,k=1, j \neq k}^n\{z \in \mathbb{K}: |z-a_{j, j}||z-a_{k,k}|\leq h_j(A)h_k(A)\}.
		\end{align*}
		\item 	If  $B=[b_{j,k}]_{1\leq j\leq n, 1\leq k \leq n}\in  \mathbb{M}_n(\mathbb{K})$ satisfies
		\begin{align*}
			|b_{j,j}||b_{k,k}|>	h_j(B)h_k(B), \quad \forall 1\leq j, k \leq n, j \neq k,
		\end{align*}
		then $B$ is invertible. 
	\end{enumerate}
\end{theorem}
\begin{proof}
	\begin{enumerate}[\upshape(i)]
		\item $ \implies $ (ii)	Let $B=[b_{j,k}]_{1\leq j\leq n, 1\leq k \leq n}\in  \mathbb{M}_n(\mathbb{K})$ satisfies
		\begin{align}\label{PI1}
			|b_{j,j}||b_{k,k}|>	h_j(B)h_k(B), \quad \forall 1\leq j, k \leq n, j \neq k.
		\end{align}
		We need to show that $B$ is invertible. Let us assume that $B$ is not invertible. Then $0\in \sigma(B)$. By assumption (i), there exist $1\leq j, k \leq n$ with $j \neq k$ such that 
		\begin{align}\label{PI2}
			|0-b_{j, j}||0-b_{k,k}|\leq h_j(B)h_k(B).
		\end{align}
		Inequalities (\ref{PI1}) and (\ref{PI2}) contradict each other. Hence $B$ is invertible. 
		\item $ \implies $ (i) Let $A=[a_{j,k}]_{1\leq j\leq n, 1\leq k \leq n}\in  \mathbb{M}_n(\mathbb{K})$ and $\lambda \in \sigma(A)$. We claim that 
		\begin{align*}
		|\lambda-a_{j, j}||\lambda-a_{k,k}|\leq h_j(A)h_k(A), \quad \text{ for some } 1\leq j, k \leq n, j \neq k.	
		\end{align*}
	Let us suppose that claim fails. Then 
	\begin{align}\label{CI}
		|\lambda-a_{j, j}||\lambda-a_{k,k}|> h_j(A)h_k(A), \quad \forall  1\leq j, k \leq n, j \neq k.
	\end{align}
Let $I_n$ be the identity matrix in $ \mathbb{M}_n(\mathbb{K})$. Define $B\coloneqq \lambda I_n-A=:[b_{j,k}]_{1\leq j\leq n, 1\leq k \leq n}$. Then $0   \in \sigma(B)$, hence $B$ is not invertible. Note that $h_j(A)=h_j(B)$ for all $1\leq j \leq n$. But we also have from (\ref{CI})
\begin{align*}
	|b_{j,j}||b_{k,k}|> h_j(B)h_k(B), \quad \forall  1\leq j, k \leq n, j \neq k.
\end{align*}
Assumption (ii) says that $B$ is invertible which is not possible. Hence claim holds. 
	\end{enumerate}
\end{proof}
By considering the transpose of a matrix, we easily get following result. 
\begin{theorem} \label{E2}
	Let $n \in \mathbb{N}$.  Following two statements are equivalent. 
	\begin{enumerate}[\upshape(i)]
		\item Let $A=[a_{j,k}]_{1\leq j\leq n, 1\leq k \leq n}\in  \mathbb{M}_n(\mathbb{K})$. Then 
		\begin{align*}
			\sigma(A)\subseteq \bigcup_{j,k=1, j \neq k}^n\{z \in \mathbb{K}: |z-a_{j, j}||z-a_{k,k}|\leq v_j(A)v_k(A)\}.
		\end{align*}
		\item 	If  $B=[b_{j,k}]_{1\leq j\leq n, 1\leq k \leq n}\in  \mathbb{M}_n(\mathbb{K})$ satisfies
		\begin{align*}
			|b_{j,j}||b_{k,k}|>	v_j(B)v_k(B), \quad \forall 1\leq j, k \leq n, j \neq k,
		\end{align*}
		then $B$ is invertible. 
	\end{enumerate}
\end{theorem}

\section{Acknowledgments}
This paper has been partly developed at the Lodha Mathematical Sciences Institute (LMSI), Mumbai, India,  where the author attended the ``Educational Workshop on High Dimensional Expanders'' from 23-25 December 2025. The author thanks the LMSI and its creators for the opportunity given to him. Author  thanks Chaitanya G. K. for making him aware of  reference \cite{LILI}. 

 \bibliographystyle{plain}
 \bibliography{reference.bib}

\end{document}